\documentclass{amsart}

\usepackage [latin1] {inputenc}

\usepackage {mathrsfs, amsmath,amsthm,  amssymb}
\usepackage{dsfont}
\usepackage [intlimits] {empheq}
\usepackage{paralist}

\usepackage{genyoungtabtikz}
\usepackage{subcaption}
\newtheorem{thm}{Theorem}[section]

\newtheorem{lem}[thm]{Lemma}

\newtheorem{defn}[thm]{Definition}

\renewcommand{\leq}{\leqslant}
\renewcommand{\geq}{\geqslant}

\renewcommand{\O}{\mathcal{O}}
\newcommand{\Var}{\operatorname{Var}}
\newcommand{\E}{\operatorname{E}}
\newcommand{\1}{\mathds{1}}
\newcommand{\Irr}{\operatorname{Irr}}
\renewcommand{\(}{\left(}
\renewcommand{\)}{\right)}
\newcommand{\floor}[1]{\left\lfloor#1\right\rfloor}

\def\fC{\mathfrak{C}}
\def\fM{\mathfrak{M}}

\title{The Probability of Generating the Symmetric Group}
\author{Sean Eberhard and Stefan-Christoph Virchow}

\begin{document}
\begin{abstract}
We consider the probability $p(S_n)$ that a pair of random permutations generates either the alternating group $A_n$ or the symmetric group $S_n$. Dixon (1969) proved that $p(S_n)$ approaches  $1$ as $n\to\infty$ and conjectured that $p(S_n)=1-1/n+o(1/n)$. This conjecture was verified by Babai (1989), using the Classification of Finite Simple Groups. We give an elementary proof of this result; specifically we show that $p(S_n)=1-1/n+\mathcal {O}(n^{-2+\epsilon})$. Our proof is based on character theory and character estimates, including recent work by Schlage-Puchta (2012).
\end{abstract}

\maketitle

\section{Introduction}
Let $G=A_n$ or $G=S_n$. We consider the probability
\begin{equation*}
p(G):=\frac{\#\{(\pi,\sigma)\in G\times G:~\langle\pi,\sigma\rangle\geq A_n\}}{|G|^2}
\end{equation*}
of ordered pairs $(\pi,\sigma)\in G\times G$ generating either the alternating group $A_n$ or the symmetric group $S_n$.

E. Netto \cite[p.\,90]{Ne1} conjectured that almost all pairs of elements from $S_n$ will generate either $A_n$ or $S_n$. J. D. Dixon \cite{Di2} was the first to prove Netto's conjecture. More precisely, he established that
\begin {equation*}
p(S_n)>1-\frac{2}{(\log\log n)^2}
\end{equation*}
for all sufficiently large $n$. Dixon conjectured that the term $2/(\log\log n)^2$ can be replaced by one of order $1/n$. J. Bovey and A. Williamson \cite{Bo1} improved Dixon's estimate to
\begin{equation*}
p(S_n)>1-\exp(-\sqrt{\log n}).
\end{equation*}
This was subsequently amended by Bovey~\cite{Bo2} to
\[
p(S_n) > 1-n^{-1+o(1)}.
\]
Finally, L. Babai \cite{Ba2} proved Dixon's conjecture and showed that
\begin{equation*}
p(S_n)=1-\frac{1}{n}+\mathcal{O}(n^{-2})
\end{equation*}
for all sufficiently large $n$.\\
In 2005, Dixon \cite{Di3} established an even better asymptotic formula for $p(S_n)$ and for $p(A_n)$. For $m\in\mathbb{N}$ the asymptotic formula is
\begin{equation*}
p(S_n)=1+\frac{c_1}{n}+\frac{c_2}{n^2}+\ldots+\frac{c_m}{n^m}+\mathcal{O}(n^{-(m+1)}),
\end{equation*}
where the coefficients $c_m$ are effectively computable. The same expansion holds for $p(A_n)$. The expansion begins
\begin{equation*}
p(S_n) = 1-\frac{1}{n}-\frac{1}{n^2}-\frac{4}{n^3}-\frac{23}{n^4}-\frac{171}{n^5}-\frac{1542}{n^6}+\mathcal{O}(n^{-7}).
\end{equation*}
See \cite{oeis} for more terms.

Babai's proof of Dixon's conjecture and Dixon's preceding asymptotic formulas for $p(S_n)$ and $p(A_n)$ rest on consequences of the Classification of Finite Simple Groups (CFSG). Babai points out \cite[Remark 1]{Ba2} that it would be desirable to find an elementary proof of Dixon's conjecture. Our aim is to give such an elementary proof. \textit{Our methods are based on character estimates and recent work by J.-C. Schalge-Puchta \cite{Sc1}  and do not need the Classification of Finite Simple Groups.} Our main results are

\begin{thm}
\label{1-1-1}
Let $\epsilon>0$. Then we have
\begin{equation*}
p(S_n)=1-\frac{1}{n}+\mathcal{O}\left(n^{-2+\epsilon}\right)
\end{equation*}
for all sufficiently large $n$.
\end{thm}

\begin{thm}
\label{1-1-2}
Let $\epsilon>0$. Then we get
\begin{equation*}
p(A_n)=1-\frac{1}{n}+\mathcal{O}\left(n^{-2+\epsilon}\right)
\end{equation*}
for all sufficiently large $n$.
\end{thm}

The main challenge in proving these Theorems is bounding the probability that a pair of random permutations generates a primitive subgroup other than $A_n$ or $S_n$. In \cite{Di3} Dixon gave an asymptotic series for the proportion of pairs generating a transitive subgroup, and in \cite{Di2} he proved that the proportion of pairs generating a transitive, imprimitive subgroup is $\leq n 2^{-n/4}$, so all that is left is to bound
\[
  P_2(n) = P(\{(\pi,\sigma) \in S_n^2 : \pi, \sigma\in H~\text{for some primitive}~H\not\geq A_n\}),
\]
where $P$ denotes the uniform distribution on $S_n\times S_n$. Using CFSG, Babai~\cite{Ba2} proved that $P_2(n) \leq n^{\sqrt{n}}/n!$. Without CFSG, Bovey~\cite{Bo2} proved that $P_2(n) \leq n^{-1+o(1)}$. We will improve this to $P_2(n) \leq n^{-2+o(1)}$: once we have this then Theorems~\ref{1-1-1} and \ref{1-1-2} follow from the above mentioned results and \cite[Theorem 2]{Di3}.

\begin{thm}
$P_2(n) \leq n^{-2+o(1)}$ as $n\to\infty$.
\end{thm}

Let us briefly outline the proof. The main insight, borrowed from Schlage-Puchta~\cite{Sc1}, is that if $\pi$ and $\sigma$ are random then the $N$ elements $\pi,\pi\sigma,\dots,\pi\sigma^{N-1}$ are approximately pairwise independent. Thus we can use the second moment method to show that there is some $i$ such that $\pi\sigma^i \in \fC$, where
\[
  \fC = \{\pi\in S_n: \exists ~p\in\Pi_n\text{ such that }\pi\text { contains a } p\text{-cycle}\},
\]
and $\Pi_n$ is the set of all primes $p$ in the range $n/2 < p < 3n/5$. Some power of $\pi\sigma^i$ is then a $p$-cycle, and we can apply the following classical result of C.~Jordan (see \cite[Theorem 3.3E]{Di1} or \cite[Theorem 13.9]{Wi2}).

\begin{lem}
\label{jordan-p-cycle}
Let $H$ be a primitive subgroup of $S_n$. Suppose that $H$ contains at least one permutation which is a $p$-cycle for a prime $p\leq n-3$. Then either $H=S_n$ or $H=A_n$.
\end{lem}

\noindent
Additionally, as something of technical trick, we will use the concept of minimal degree. Recall that the \emph{minimal degree} $m(H)$ of a non-trivial subgroup $H\leq S_n$ is the minimal number of points moved by a non-identity element of $H$. The following bound is due to Babai (see \cite[Theorem~5.3A and Theorem~5.4A]{Di1}).

\begin{lem}
\label{babai-min-degree}
Let $H<S_n$ be a primitive permutation group not containing $A_n$. Then $m(H)>\sqrt{n}/2$.
\end{lem}
\noindent
This Lemma allows us to restrict $\sigma$ to the set
\[
  \fM = \{\sigma \in S_n : m(\langle \sigma\rangle) > \sqrt{n}/2\},
\]
which slightly boosts the approximate pairwise independence of $\pi,\pi\sigma,\dots,\pi\sigma^{N-1}$. (A bound of the form $m(H) > c\sqrt{n}/\log n$ due to Jordan would also suffice for us.)
\noindent
To bound the variance in the second moment method we use character theory, and thus the proof comes down to a certain bound in terms of characters. Finally, we apply a character bound due to Müller and Schlage-Puchta \cite{Mu1} to conclude.

We can use basically the same method to bound
\[
  P_3(n) = P(\{(\pi,\sigma,\tau)\in S_n^3 : \pi, \sigma,\tau\in H~\text{for some primitive}~H\not\geq A_n\}),
\]
and in this case we have significantly more leverage as we can consider the collection of all words of the form $\pi w(\sigma,\tau)$ with $w$ a short word in two letters. Again we have approximate pairwise independence, so again we can use the second moment method. This idea leads to the following bound.

\begin{thm}
$P_3(n) \leq \exp(-cn^{1/3})$ as $n\to\infty$.
\end{thm}

\noindent
By combining this with \cite[Section~4]{Di3} we have
\[
  P(\langle \pi,\sigma,\tau\rangle\geq A_n) = 1 - \frac1{n^2} - \frac{3}{n^4} - \frac{6}{n^5} + \O(n^{-6})
\]
(and more terms can be mechanically computed).

Finally, we should mention
\[
  P_1(n) = P(\{\pi \in S_n : \pi\in H~\text{for some primitive}~H\not\geq A_n\}).
\]
On CFSG it is known that $P_1(n) \leq n^{-1 + o(1)}$ (see \cite[Theorem~1.3]{Eb1}), and this is the best possible bound which depends only on the crude size of $n$. The best CFSG-free bound is still $P_1(n) \leq |\fM|/n! = n^{-1/2 + o(1)}$ due to Bovey~\cite{Bo2}.

\section{Some character theory}

In this section we review some results from character theory which are essential for our proof. We denote by $\Irr(S_n)$ the set of irreducible characters of $S_n$. For a conjugacy class $C$ of $S_n$ and $\chi\in\Irr(S_n)$ we write $\chi(C)$ to denote $\chi(\pi)$ for $\pi\in C$. We write $\langle\cdot,\cdot\rangle$ for the usual inner product on the space $\mathbb{C}^{S_n}$, i.e.,
\[
  \langle f, g\rangle = \frac1{n!} \sum_{\pi \in S_n} f(\pi)\overline{g(\pi)}.
\]

\begin{lem}\label{convolution-lemma}
Let $C_1$ and $C_2$ be conjugacy classes of $S_n$ and let $\tau\in S_n$. Then
\begin{equation*}
\#\{(x,y)\in C_1\times C_2:~xy=\tau\}=\frac{|C_1||C_2|}{n!}\sum_{\chi\in \operatorname{Irr}(S_n)}\frac{\chi(C_1)\chi(C_2)\chi(\tau^{-1})}{\chi (1)}.
\end{equation*}
Thus if $C_1$, $C_2$, and $C_3$ are conjugation-invariant subsets of $S_n$ we have
\[
  \#\{(x,y) \in C_1 \times C_2 : xy \in C_3\} = n!^2 \sum_{\chi\in\Irr(S_n)} \frac{\langle \chi, \1_{C_1}\rangle \langle \chi, \1_{C_2}\rangle \langle \chi, \1_{C_3}\rangle}{\chi(1)}.
\]
\end{lem}
\begin{proof} For the first equation, see \cite[Proposition 9.33]{Cu1} or \cite[Theorem 6.3.1]{Ke1}. The second equation follows from partitioning $C_1$ and $C_2$ into conjugacy classes and adding.
\end{proof}

Recall that the irreducible characters of $S_n$ are explicitly parameterized by \emph{partitions} of $n$, or sequences $\lambda = (\lambda_1,\lambda_2,\dots,\lambda_l)$, where $\lambda_1 \geq \cdots \geq \lambda_l$ are positive integers such that $\lambda_1 + \cdots + \lambda_l = n$. We write $\lambda \vdash n$ to indicate that $\lambda$ is a partition of $n$, and we write $\chi^\lambda$ for the irreducible character of $S_n$ corresponding to $\lambda$. The \emph{Ferrers diagram} of $\lambda$ is an array of $n$ boxes having $l$ left-justified rows with row $i$ containing $\lambda_i$ boxes for $1\leq i\leq l$. We write $(i,j)\in\lambda$ to indicate that $(i,j)$ is a box in row $i$ and column $j$ in the Ferrers diagram of $\lambda$.

We shall apply the \emph{Murnaghan--Nakayama rule.} 

\begin{defn}
Let $\lambda\vdash n$ be a partition. A \emph{rim hook} $h$ is an edgewise connected part of the Ferrers diagram of $\lambda$, obtained by starting from a box at the right end of a row and at each step moving downwards or leftwards only, which can be removed to leave a proper Ferrers diagram denoted by $\lambda\backslash h$. An $r$\emph{-rim hook} is a rim hook containing $r$ boxes.\\
The \emph{leg length} of a rim hook $h$ is
\begin{equation*}
ll(h):=(\text{the number of rows of }h)-1.
\end{equation*}
Let $\pi\in S_n$ be a permutation with cycle type $(1^{\alpha_1},\ldots,q^{\alpha_q},\ldots,n^{\alpha_n})$ and $\alpha_q\geq 1$. Denote $\pi\backslash q\in S_{n-q}$ a permutation with cycle type $(1^{\alpha_1},\ldots,q^{\alpha_q-1},\ldots,(n-q)^{\alpha_{n-q}})$.
\end{defn}

\begin{lem}[Murnaghan--Nakayama Rule]
\label{1-4-2}
Let $\lambda\vdash n$ be a partition. Suppose that $\pi\in S_n$ is a permutation which contains a $q$-cycle. Then we have
\begin{equation*}
\chi^{\lambda}(\pi)=\sum_{\substack{h\\q\text{-rim hook}\\\text{of }\lambda}}(-1)^{ll(h)}\chi^{\lambda\backslash h}(\pi\backslash q).
\end{equation*}
\end{lem}
\begin{proof} See \cite[\S 9]{Na1} or \cite[Theorem 4.10.2]{Sa1}.\end{proof}

The dimension $\chi^\lambda(1)$ of the irreducible representation associated with $\lambda$ can be computed via the \emph{hook formula.} 

\begin{defn}
Let $\lambda\vdash n$ be a partition. The \emph{hook} of $(i,j)\in\lambda$ is
\begin{equation*}
H_{i,j}(\lambda):=\{(i,j')\in\lambda:~j'\geq j\}\cup\{(i',j)\in\lambda:~i'\geq i\}.
\end{equation*}
\end{defn}

\begin{lem}[Hook Formula]
\label{1-4-3}
Let $\lambda\vdash n$ be a partition. Then
\begin{equation*}
\chi^{\lambda}(1)=\frac{n!}{\prod\limits_{(i,j)\in \lambda}|H_{i,j}(\lambda)|}.
\end{equation*}
\end{lem}
\begin{proof} See \cite[Theorem 1]{Fr1} or \cite[Theorem 3.10.2]{Sa1}.\end{proof}

We combine the Murnaghan--Nakayama rule and hook formula to show that $\chi(1)$ is exponentially large whenever $\chi$ is nontrivial and $\langle\chi,\1_\fC\rangle \neq 0$.

Note that $\mathfrak{C}$ is the union of conjugacy classes from $S_n$, since a conjugacy class consists of all permutations with the same cycle type. For fixed $p\in\Pi_n$ let $C_1,...,C_s$ denote all conjugacy classes of $\mathfrak{C}$ which contain a $p$-cycle. By removing a $p$-cycle from $C_i$ we obtain a conjugacy class $C_{i}\!\setminus p$ from $S_{n-p}$. Apparently, we have $S_{n-p}=\dot{\bigcup}_{i=1,...s}C_{i}\!\setminus p$. In addition, computing the cardinality of $C_{i}$ and $C_{i}\!\setminus p$ (see \cite[Formula (1.2)]{Sa1}) we obtain 
\begin{equation*}
|C_{i}|=\frac{n!}{(n-p)!p}|C_{i}\!\setminus p|. 
\end{equation*}
Let $\lambda\vdash n,~\lambda\neq (n)$. We now apply the Murnaghan--Nakayama rule (Lemma \ref{1-4-2}):

\begin{align}
\sum_{1\leq i\leq s}|C_i|\chi^{\lambda}(C_i)&=\sum_{1\leq i\leq s}\frac{n!}{(n-p)!p}|C_{i}\!\setminus p|\sum_{\substack{h\\p\text{-rim hook}\\\text{of }\lambda}}(-1)^{ll(h)}\chi^{\lambda\backslash h}(C_{i}\!\setminus p)\notag\\
&=\sum_{\substack{h\\p\text{-rim hook}\\\text{of }\lambda}}(-1)^{ll(h)}\cdot\frac{n!}{p}\cdot\langle\chi^{(n-p)},\chi^{\lambda\backslash h}\rangle \notag\\
&=
\begin{cases}
(-1)^{ll(h)}\frac{n!}{p},&\text{if }\exists~p\text{-rim hook } h \text{ of }\lambda\text{ with }\lambda\backslash h=(n-p),\\
0, &\text{otherwise.}
\end{cases}\nonumber
\end{align}
Thus we can have $\langle \chi^\lambda, \1_\fC\rangle \neq 0$ only if $\lambda \in \Lambda_{n,p}$ for some $p\in \Pi_n$, where
\[
  \Lambda_{n,p} = \{\lambda\vdash n:~\lambda\neq (n)~\text{and}~\exists~p\text{-rim hook}~h~\text{such that}~\lambda\backslash h=(n-p)\}.
\]

\begin{lem}
  \label{dim-chi}
Let $n$ be sufficiently large. If $p\in\Pi_n$ and $\lambda\in \Lambda_{n,p}$ then
\[
  \chi^\lambda(1) \geq \exp(n/4).
\]
\end{lem}

\begin{proof}
\def\casea{{\rm(a)}}
\def\caseb{{\rm(b)}}

First, we investigate the set $\Lambda_{n,p}$. We claim the following: \textit{Let $p\in\Pi_n$ and $\lambda\in \Lambda_{n,p}$. Then we have $\lambda\in\Lambda_{n,p}$ if and only if $\lambda=(\lambda_1,\lambda_2,1^{n-\lambda_1-\lambda_2})$, where either $\casea$ $~~\lambda_1=n-p$ and $1\leq \lambda_2\leq n-p~~$, or $\caseb$ $~~n-p<\lambda_1\leq p-1$ and $\lambda_2=n-p+1$.} (See Figure~\ref{young-diagrams}.)

\begin{figure}
\centering
\newcommand\ylw{\Yfillcolour{gray}}
\newcommand\wht{\Yfillcolour{white}}
\begin{tabular}{llll}
$\casea$ &
  $\Yboxdim{14pt}\young(!\ylw~~~!\wht,~,~,~,~,~,~,~)$&
  $\Yboxdim{14pt}\young(!\ylw~~~!\wht,~~,~,~,~,~,~)$&
  $\Yboxdim{14pt}\young(!\ylw~~~!\wht,~~~,~,~,~,~)$\\
  \vspace{10pt}\\
$\caseb$ &
  $\Yboxdim{14pt}\young(!\ylw~~~!\wht~,~~~~,~,~)$&
  $\Yboxdim{14pt}\young(!\ylw~~~!\wht~~,~~~~,~)$&
  $\Yboxdim{14pt}\young(!\ylw~~~!\wht~~~,~~~~)$
\end{tabular}
\caption{The two cases of $\lambda\backslash h = (n-p)$, $\lambda\neq (n)$ for $n = 10$, $p = 7$}
\label{young-diagrams}
\end{figure}

You can see this as follows: Let $\lambda\in\Lambda_{n,p}$. Then the Ferrers diagram of $\lambda$ has a block of $n-p$ boxes in the first row and around this block there is a $p$-rim hook $h$. If the rim hook $h$ does not contain a box from the first row of $\lambda$, then $\lambda_1=n-p$ and $1\leq\lambda_2\leq n-p$, i.e., $\casea$ is satisfied. If $h$ contains a box from the first row, then since $\lambda \neq (n)$ it follows immediately that $n-p<\lambda_1$ and $\lambda_2=n-p+1$. As $h$ is a $p$-rim hook, we also have $\lambda_1\leq p-1$. So $\caseb$ is fulfilled. Conversely, if $\lambda=(\lambda_1,\lambda_2,1^{n-\lambda_1-\lambda_2})$ such that $\casea$ or $\caseb$ is satisfied, then there obviously exists a $p$-rim hook $h$ such that $\lambda\backslash h=(n-p)$. Thus $\lambda\in\Lambda_{n,p}$.

Second, let $p\in\Pi_n$ and $\lambda\in\Lambda_{n,p}$. Using the above description of $\Lambda_{n,p}$ yields, for the product of hook lengths of $\lambda$,
\begin{equation*}
 T:=\prod_{(i,j)\in\lambda}|H_{i,j}(\lambda)|\leq n\lambda_1!p(\lambda_2-1)!(n-\lambda_1-\lambda_2)!.
\end{equation*}
We can bound this expression as follows:\\
\textit{Case $\casea$:} $T\leq n(n-p)!p(\lambda_2-1)!(p-\lambda_2)!\leq n(n-p)!p!$.\\
\textit{Case $\caseb$:} $T\leq n\lambda_1!p(n-p)!(p-1-\lambda_1)!\leq n(n-p)!p!$.\\
Thus it follows from the hook formula (Lemma \ref{1-4-3}) for sufficiently large $n$ that
\begin{equation*}
\chi^{\lambda}(1)\geq \frac{1}{n}\binom{n}{p} \geq \frac{1}{n}\left(\frac{n}{p}\right)^p\geq\exp(\tfrac{1}{4}n).\qedhere
\end{equation*}
\end{proof}

Finally, we will use the following estimate, due to T. W. Müller and J.-C. Schlage-Puchta \cite[Theorem 1]{Mu1}, which improves the trivial bound  $|\chi(\sigma)|\leq \chi(1)$ for an irreducible character $\chi$ of $S_n$, if the number $f(\sigma)$ of fixed points of $\sigma\in S_n$ is not too large.

\begin{lem}
\label{character-bound}
Let $\chi\in\Irr(S_n)$ be an irreducible character, let $\sigma\in S_n$ be a permutation and let $n$ be sufficiently large. Then we have
\begin{equation*}
|\chi(\sigma)|\leq \chi(1)^{1-\delta(\sigma)}
\end{equation*}
where
\begin{equation*}
 \delta(\sigma):=
 \begin{cases}
  \frac{1}{13}, &\text{if }f(\sigma)=0\\
  \frac{\log\left(n/f(\sigma)\right)}{32 \log n}, &\text{if }1\leq f(\sigma)\leq n.
  \end{cases}
\end{equation*}

\end{lem}

\section{Two permutations}

If $\sigma\in H$ for some primitive $H\not\geq A_n$, then we know from Lemma~\ref{babai-min-degree} that $\sigma\in \fM$. Suppose then that we pick $(\pi,\sigma)\in S_n \times \fM$ uniformly at random, and let $X$ be the number of $i\in\{0,\dots,N-1\}$ such that $\pi\sigma^i \in \fC$. If $X>0$ then by Lemma~\ref{jordan-p-cycle} we cannot have $\pi,\sigma\in H$ for any primitive $H\not\geq A_n$. Thus by Chebyshev's inequality we get
\begin{equation}\label{second-moment-bound}
  P_2(n) \leq \frac{|\fM|}{n!} Q(X = 0) \leq \frac{|\fM|}{n!} \frac{\Var X}{(\E X)^2},
\end{equation}
where $Q$ denotes the uniform distribution on $S_n \times \fM$.
Now since we are still taking $\pi$ uniformly at random from $S_n$ we clearly have
\[
  \E X = N \frac{|\fC|}{n!},
\]
while
\[
  \Var X = N \frac{|\fC|}{n!} \(1 - \frac{|\fC|}{n!}\) + 2 \sum_{0\leq i < j < N} \( Q(\pi \sigma^i, \pi \sigma^j \in \fC) - \(\frac{|\fC|}{n!}\)^2\).
\]
Now we express $Q(\pi \sigma^i, \pi \sigma^j \in \fC)$ in terms of characters. Define
\[
  r_\nu(\tau) = \#\{\sigma\in\fM : \sigma^\nu = \tau\}.
\]
Then by Lemma~\ref{convolution-lemma} we have
\begin{align*}
  Q(\pi\sigma^i, \pi\sigma^j \in \fC)
  &= \frac1{n!|\fM|} \sum_{(\pi,\sigma)\in S_n\times \fM} \1_\fC(\pi\sigma^i) \1_\fC(\pi\sigma^j)\\
  &= \frac1{n!|\fM|} \sum_{x,y\in S_n} \1_\fC(x) \1_\fC(y) r_{j-i}(x^{-1} y)\\
  &= \frac{n!}{|\fM|} \sum_{\chi\in\Irr(S_n)} \frac{\langle \chi, \1_\fC\rangle^2 \langle \chi, r_{j-i}\rangle}{\chi(1)}.
\end{align*}
The contribution from the trivial character $\chi=1$ is precisely $(|\fC|/n!)^2$, since $\langle 1, r_{j-i}\rangle=|\fM|/n!$. Thus it follows
\begin{equation}\label{variance}
  \Var X = N \frac{|\fC|}{n!}\(1 - \frac{|\fC|}{n!}\) + \frac{2 n!}{|\fM|} \sum_{\nu = 1}^N (N - \nu) \sum_{\chi\neq 1} \frac{\langle \chi, \1_\fC\rangle^2 \langle \chi, r_\nu\rangle}{\chi(1)}.
\end{equation}
\noindent
Note that
\[
  \langle \chi, r_\nu\rangle = \frac1{n!} \sum_{\sigma\in\fM} \chi(\sigma^\nu).
\]
\noindent
Therefore we obtain
\begin{align*}
  \frac{|\langle \chi, r_\nu\rangle|}{\chi(1)} 
  &\leq \frac1{n!} \sum_{\sigma \in \fM} \frac{|\chi(\sigma^\nu)|}{\chi(1)}\\
  &\leq \frac{\#\{\sigma: \sigma^\nu = 1\}}{n!} + \max_{\sigma\in\fM: \sigma^\nu\neq 1} \frac{|\chi(\sigma^\nu)|}{\chi(1)}.
\end{align*}
By Lemma~\ref{dim-chi} we know that $\chi(1) \geq \exp(n/4)$ whenever $\langle \chi, \1_\fC\rangle \neq 0$, and by definition of $\fM$ we know that $\sigma^\nu$ has at most $n-n^{1/2}/2$ fixed points whenever $\sigma^\nu \neq 1$, so Lemma~\ref{character-bound} yields
\[
  \max_{\sigma\in\fM : \sigma^\nu\neq 1} \frac{|\chi(\sigma^\nu)|}{\chi(1)} \leq \exp(n/4)^{-\delta},
\]
where
\[
  \delta = \frac{\log(n/(n-n^{1/2}/2))}{32\log n}\geq \frac{n^{-1/2}}{64 \log n}.
\]
Thus
\[
  \max_{\sigma\in\fM:\sigma^\nu\neq 1} \frac{|\chi(\sigma^\nu)|}{\chi(1)} \leq \exp\(-\frac{n^{1/2}}{2^8 \log n}\).
\]
By orthogonality of characters it follows that
\begin{align}
  \sum_{\nu = 1}^N (N-\nu)& \sum_{\chi\neq 1} \frac{\langle \chi, \1_\fC\rangle^2 \langle \chi, r_\nu\rangle}{\chi(1)}\nonumber \\
  &\leq N \sum_{\nu=1}^N \sum_{\chi\in\Irr(S_n)} |\langle \chi, \1_\fC\rangle|^2 \(\frac{\#\{\sigma: \sigma^\nu = 1\}}{n!} + \exp\(- \frac{n^{1/2}}{2^8 \log n}\)\)\nonumber \\
  &= N \frac{|\fC|}{n!} \(\sum_{\nu = 1}^N \frac{\#\{\sigma : \sigma^\nu = 1\}}{n!} + N \exp\(-\frac{n^{1/2}}{2^8\log n}\)\).\label{sigma-nu-bound}
\end{align}
To finish we need to count pairs $(\sigma,\nu)$ such that $\sigma^\nu = 1$.

We will need the following simple bound for the number of permutations without long cycles. (See \cite{Ma1,Pe1} for more precise estimates involving the Dickman function.)

\begin{lem}
Let $r$ and $m$ be positive integers such that $r \leq m/2$. Then the number of $\pi\in S_m$ all of whose cycles have length at most $r$ is bounded by
\[
  \(\frac{2r}{m}\)^\frac{m}{2r} m!.
\]
\end{lem}
\begin{proof}
Let $p(m,r)$ be the probability that a random $\pi\in S_m$ has no cycle of length greater than $r$. Recall that we can sample $\pi$ as follows: First we choose the length $j$ of the cycle containing $1$ uniformly from $\{1,\dots,m\}$, then we choose the set $\{\pi(1),\dots,\pi^{j-1}(1)\}$ uniformly from all possible $(j-1)$-subsets of $\{2,\dots,m\}$, and then we choose (inductively) a random permutation of $\{1,\pi(1),\dots,\pi^{j-1}(1)\}^c$. Since the probability that $j\leq r$ is clearly $r/m$, we deduce the recurrence
\[
  p(m,r) \leq \frac{r}{m} \max_{1\leq j \leq r} p(m-j,r).
\]
Let $q(m,r) = \max_{m' \geq m} p(m',r)$. Then we have
\[
  q(m,r) \leq \frac{r}{m} q(m-r,r),
\]
whenever $m > r$, while of course $q(m,r)=1$ if $m \leq r$. Thus provided $r \leq m/2$ we have
\[
  q(m,r) \leq \frac{r}{m} \frac{r}{m-r} \cdots \frac{r}{m - \floor{m/(2r)} r} \leq \(\frac{2r}{m}\)^{1 + \floor{\frac{m}{2r}}} \leq \(\frac{2r}m\)^\frac{m}{2r}.\qedhere
\]
\end{proof}

\begin{lem}
Assume $N\geq n$. Then we have
 \begin{equation*}
  k(N) := \#\{(\nu,\sigma): 1\leq\nu\leq N,~ \sigma\in S_n,~\sigma^{\nu}=1 \} = N^{1+o(1)} n! n^{-2}.
 \end{equation*}
for all sufficiently large $n$.
\end{lem}

In the proof we will find it convenient to use the Vinogradov notation $X \ll Y$ familiar from analytic number theory, which means simply $X\leq CY$ for some implicit constant $C$, or in other words $X \leq O(Y)$.

\begin{proof}
First, we establish that $k(N)\gg N n! n^{-2}$: Let $D$ be the conjugacy class of $n$-cycles in $S_n$. Obviously, $|D|=\frac{n!}{n}$ and $\textrm{ord}(\sigma)=n$ for $\sigma\in D$. Thus we obtain
\begin{equation*}
 k(N)\geq \frac{n!}{n}\cdot\left\lfloor \frac{N}{n}\right\rfloor \geq \frac12 N n! n^{-2}.
\end{equation*}

Second, we prove that $k(N)\leq N^{1+o(1)} n!n^{-2}$: For a permutation $\pi\in S_m$ and $1\leq j\leq m$ denote by $c_j(\pi)$ the number of $j$-cycles of $\pi$. Let $m$ be sufficiently large and let $r(m):=\left\lfloor\frac{m}{\log m}\right\rfloor$. By the previous Lemma we have
\begin{align*}
A_1(m)
&:=\#\{\sigma\in S_m:~~c_j(\sigma)=0 ~\forall~ j>r(m)\}\\
& \leq \(\frac{2r(m)}m\)^\frac{m}{2r(m)} m!\\
& \leq \(\frac{2}{\log m}\)^\frac{\log m}{2} m!\\
& \ll \frac{m!}{m^2}.
\end{align*}
In addition, we consider for a given positive integer $\nu$ the number of permutations $\sigma\in S_m$ having at least one cycle of length $>r(m)$ such that $\sigma^{\nu}=1$:
\begin{align*}
 A_2(m,\nu) &:=\#\left\{\sigma\in S_m:~~\sigma^{\nu}=1\wedge\bigl(\exists j>r(m):~c_j(\sigma)\neq 0\bigr)\right\}\\
 &\leq \sum_{\substack{j|\nu\\j>r(m)}}\frac{m!}{j}\leq \frac{m!}{r(m)}\cdot d(\nu),
\end{align*}
where $d(\nu)$ denotes the number of divisors of $\nu$. Combining the previous two results yields
\begin{equation*}
 A_3(m,\nu):=\#\{\sigma\in S_m:~~\sigma^{\nu}=1\}\leq A_1(m)+A_2(m,\nu)\ll \frac{m!}{r(m)}\cdot d(\nu).
 \end{equation*}
Furthermore, we give an upper bound for the sum $\sum_{1\leq\nu\leq N}A_2(n,\nu)$. Applying the preceding estimate we get for all sufficiently large $n$
\begin{align*}
 \sum_{1\leq\nu\leq N}A_2(n,\nu)
 &\leq \sum_{1\leq\nu\leq N}\sum_{\substack{r(n)<j\leq n\\j|\nu}}\binom{n}{j}\frac{j!}{j} A_3(n-j,\nu)\\
 &\ll  N n! n^{-2} + \sum_{r(n)<j<n}\sum_{\substack{1\leq\nu\leq N\\j|\nu}}\frac{n!\log (n-j)}{j(n-j)}\cdot d(\nu).
\end{align*}
As $d(\nu)\ll \nu^{1/\log\log\nu}$ (see \cite[Theorem 317]{Ha1}) we have $d(\nu)\leq N^{o(1)}$ for $\nu\leq N$. Therefore, it follows that
\begin{align*}
\sum_{1\leq\nu\leq N}A_2(n,\nu)
 &\ll N n! n^{-2} + N^{1+o(1)} n! \sum_{r(n) < j < n} \frac{\log (n-j)}{j^2(n-j)}\\
 & = N^{1+o(1)} n! n^{-2}.
\end{align*}
Thus, we conclude
\begin{equation*}
 k(N)\leq \sum_{1\leq\nu\leq N}\( A_1(n)+A_2(n,\nu)\)\leq N^{1+o(1)} n!n^{-2}. \qedhere
\end{equation*}
\end{proof}
\noindent
Combining the preceding Lemma with \eqref{variance} and \eqref{sigma-nu-bound} we get
\[
  \Var X \leq N \frac{|\fC|}{n!} + N^{2+o(1)} \frac{|\fC|}{|\fM|} n^{-2}.
\]
\noindent
For each $p\in\Pi_n $ there are $\binom{n}{p}(p-1)!(n-p)!=\frac{n!}{p}$ elements of $S_n$ containing a $p$-cycle. Therefore we have $|\fC| \geq \frac{n!}{n}|\Pi_n|$ and (a weak version of) the prime number theorem yields
\[
 |\fC| \geq \frac{n!}{2\log n}
\]
for sufficiently large $n$. Thus it follows from \eqref{second-moment-bound} that
\[
  P_2(n) \leq \frac1N \frac{|\fM|}{|\fC|} + N^{o(1)} \frac{n!}{|\fC|} n^{-2} \leq \frac{2 \log n} N + N^{o(1)} \frac{\log n}{n^2}.
\]
Putting $N = n^2$ we get
\[
  P_2(n) \leq n^{-2+o(1)},
\]
as required.

\section{Three permutations}

In this last section we consider
\[
  P_3(n) = P(\{(\pi,\sigma,\tau) \in S_n^3 : \pi,\sigma,\tau \in H~\text{for some primitive}~H\not\geq A_n\}).
\]
The proof is much like that of the previous section, except that we use the collection of words of the form $\pi w(\sigma,\tau)$ in place of $\pi,\pi\sigma,\dots,\pi\sigma^{N-1}$.

Let
\[
  \fM_2 = \{(\sigma,\tau)\in S_n^2 : m(\langle \sigma,\tau\rangle) > \sqrt{n}/2\}.
\]
By Lemma~\ref{babai-min-degree} we know that if $\sigma,\tau\in H$ for some primitive $H\not\geq A_n$ then $(\sigma,\tau)\in\fM_2$, so we may assume that we pick $(\sigma,\tau)$ randomly from $\fM_2$. Let $W_N$ be the set of all words $w\in F_2$ of length at most $N$. Supposing we pick $\pi \in S_n$ and $(\sigma,\tau)\in\fM_2$ at random, let $X$ be the number of $w\in W_N$ of length at most $N$ such that $\pi w(\sigma,\tau)\in\fC$. Then
\[
  P_3(n) \leq \frac{|\fM_2|}{n!^2} Q(X=0) \leq \frac{|\fM_2|}{n!^2} \frac{\Var X}{(\E X)^2} ~,
\]
where $Q$ denotes the uniform distribution on $S_n\times \fM_2$. Now
\[
  \E X = |W_N| \frac{|\fC|}{n!}
\]
and
\begin{align*}
  \Var X
  &= \sum_{w,w'\in W_N} \left(Q(\pi w(\sigma,\tau), \pi w'(\sigma,\tau) \in \fC) - \frac{|\fC|^2}{n!^2}\right)\\
  &= |W_N| \frac{|\fC|}{n!}\(1 - \frac{|\fC|}{n!}\) + \sum_{\substack{w,w'\in W_N\\ w\neq w'}} \( Q(\pi w(\sigma,\tau), \pi w'(\sigma,\tau) \in \fC) - \frac{|\fC|^2}{n!^2}\).
\end{align*}
For $w \in F_2$ and $x\in S_n$, let
\[
  r_w(x) = \#\{(\sigma,\tau)\in\fM_2 : w(\sigma,\tau) = x\}.
\]
Then
\begin{align*}
  \frac{|\fM_2|}{n!^2} Q(\pi w(\sigma,\tau), \pi w'(\sigma,\tau) \in \fC)
  &= \frac1{n!^3} \sum_{\pi\in S_n} \sum_{(\sigma,\tau)\in\fM_2} \mathds{1}_\fC(\pi w(\sigma,\tau)) \mathds{1}_\fC(\pi w'(\sigma,\tau))\\
  &= \frac1{n!^3} \sum_{x,y\in S_n} \1_\fC(x) \1_\fC(y) r_{w^{-1} w'}(x^{-1} y),\\
  &= \frac1{n!} \sum_{\chi\in\Irr(S_n)} \frac{\langle \chi, \1_\fC\rangle^2 \langle \chi, r_{w^{-1} w'}\rangle}{\chi(1)}.
\end{align*}
The contribution from the trivial character $\chi = 1$ is precisely
\[
  \frac{|\fM_2||\fC|^2}{n!^4}.
\]
To bound the other terms note that
\begin{align*}
  \frac1{n!} \frac{\langle \chi, r_w\rangle}{\chi(1)}
  &= \frac1{n!^2} \sum_{(\sigma,\tau)\in\fM_2} \frac{\chi(w(\sigma,\tau))}{\chi(1)}\\
  &= \frac{\#\{(\sigma,\tau)\in\fM_2: w(\sigma,\tau)=1\}}{n!^2} + \frac1{n!^2} \sum_{\substack{(\sigma,\tau)\in\fM_2\\ w(\sigma,\tau)\neq 1}} \frac{\chi(w(\sigma,\tau))}{\chi(1)},
\end{align*}
so
\[
  \frac1{n!} \frac{|\langle \chi,r_w\rangle|}{\chi(1)} \leq \frac{\#\{(\sigma,\tau)\in S_n^2 : w(\sigma,\tau) = 1\}}{n!^2} + \max_{\substack{(\sigma,\tau)\in\fM_2\\w(\sigma,\tau)\neq 1}} \frac{|\chi(w(\sigma,\tau))|}{\chi(1)}.
\]
Provided that $\chi\neq 1$ and $\langle \chi,\1_\fC\rangle \neq 0$, the second term is bounded by
\[
  \exp(-cn^{1/2}/\log n),
\]
just as in the previous section. The first term is also small, by the following Lemma (see \cite[Lemma~2.2]{Eb2}).

\begin{lem}
Let $w\in F_2$ be a non-trivial word of length at most $k \leq \sqrt{n/2}$. If $\sigma,\tau\in S_n$ are chosen uniformly at random then the probability that $w(\sigma,\tau) = 1$ is bounded by $\exp(-cn/k^2)$.
\end{lem}

Thus provided $w\neq 1$ and $N\leq c n^{1/3}$ we have
\[
  \frac1{n!} \frac{|\langle \chi, r_w\rangle|}{\chi(1)} \leq \exp(-cn^{1/3}).
\]
Therefore, it follows for $w\neq w'$ that
\begin{align*}
  \left|\frac1{n!} \sum_{\substack{\chi\in\Irr(S_n)\\\chi\neq 1}} \frac{\langle \chi,\1_\fC\rangle^2 \langle \chi, r_{w^{-1}w'}\rangle}{\chi(1)} \right|
  &\leq \sum_{\chi\in\Irr(S_n)} |\langle \chi, \1_\fC\rangle|^2 \exp(-cn^{1/3})\\
  & = \frac{|\fC|}{n!} \exp(-cn^{1/3})\\
  &\leq \exp(-cn^{1/3}).
\end{align*}
Thus
\[
  \frac{|\fM_2|}{n!^2} \Var X \leq |W_N|\frac{|\fC|}{n!} + |W_N|^2 \exp(-cn^{1/3}),
\]
so we deduce
\[
  P_3(n) \leq \frac{\O(\log n)}{|W_N|} + \exp(-cn^{1/3}).
\]
Note that $|W_N| = 4 \cdot 3^{N-1}$. Taking $N = \lfloor cn^{1/3}\rfloor$, we conclude
\[
  P_3(n) \leq \exp(-cn^{1/3}).
\]

\section*{Acknowledgements}
I, Stefan-Christoph,  would like to express my deep gratitude to Jan-Christoph Schlage-Puchta for his proposal to consider this theme and for the many inspiring discussions we had. Furthermore, I would like to offer my special thanks to Andrzej Zuk for his interest in the subject. Moreover, I am very grateful to the referees for their helpful suggestions. Finally, I would like to express my very great appreciation to my family for their support and encouragement throughout my study.

\par\bigskip\noindent
\textbf{Author information}\\
\textsc{Sean Eberhard}, London, UK\\
E-mail: eberhard.math@gmail.com
\par\bigskip
\noindent
\textsc{Stefan-Christoph Virchow}, Institut für Mathematik, Universität Rostock\\
Ulmenstr. 69 Haus 3, 18057 Rostock, Germany\\
E-mail: stefan.virchow@uni-rostock.de

\end{document}